\documentclass[reqno,11pt]{amsart}

\usepackage{amssymb,amsfonts,amsmath,amsthm,setspace}
\usepackage{fullpage}
\newtheorem{theorem}{Theorem}[section]
\newtheorem{corollary}[theorem]{Corollary}
\newtheorem{lemma}[theorem]{Lemma}
\newtheorem{proposition}[theorem]{Proposition}

\theoremstyle{definition}
\newtheorem{definition}[theorem]{Definition}

\newtheorem{example}[theorem]{Example}

\theoremstyle{remark}
\newtheorem{remark}[theorem]{Remark}

\newcommand{\N}{\mathbb{N}}
\newcommand{\Q}{\mathbb{Q}}
\newcommand{\R}{\mathbb{R}}
\renewcommand{\d}{{~\mathrm d}}
\def\esup{\mathop\mathrm{ess\,sup\,}}
\newcommand{\supp}{\mathop{\rm supp}\nolimits}
\newcommand{\spt}{\mathop{\rm spt}\nolimits}
\newcommand{\wstar}{\stackrel{*}{\rightharpoonup}}
\newcommand{\Pb}{\mathcal P}
\newcommand{\restr}[1]{\lower3pt\hbox{$_{|{#1}}$}}
\newcommand{\iub}{{\bf \underline{i}}}
\newcommand{\jj}{\mathbf{j}}
\newcommand{\xbar}{\overline{x}}
\newcommand{\mcal}{\mathcal M}

\title[Sufficiency of $c$-cyclical monotonicity]{Sufficiency of $c$-cyclical monotonicity in a class of multimarginal optimal transport problems}
\author{Luigi De Pascale}
\address{Dipartimento di Matematica e Informatica, Universit\`a di Firenze \\ Viale Morgagni 67/a, 50134 Firenze \\ Italy}
\email{luigi.depascale@unifi.it}
\author{Anna Kausamo}
\address{Dipartimento di Matematica e Informatica, Universit\`a di Firenze \\ Viale Morgagni 67/a, 50134 Firenze \\ Italy}
\email{akausamo@gmail.com}

\begin{document}
\maketitle

\begin{abstract}
$c$-cyclical monotonicity is the most important optimality condition for an optimal transport plan. While the proof of necessity is relatively easy, the proof of sufficiency is 
often more difficult or even elusive. 
We present here a new approach and we show how known results are derived in this new framework and how this approach allows to prove sufficiency in situations previously not treatable. 
\end{abstract}

\section{Introduction}
\subsection{Optimal transport problems}

Consider $N\ge 2$  Polish probability spaces  $(X^1,\mu^1),\ldots,\\ (X^N,\mu^N)$ and $X=\prod_{i=1}^NX^i$. Let $c:X\to [0,\infty]$ and consider the  \emph{Multi-marginal optimal transport} (MOT) problems 
\[
 \min_{\gamma \in \Pi(\mu^1,\ldots,\mu^N)}C[\gamma]:= \min_{\gamma \in \Pi(\mu^1,\ldots,\mu^N)}\int_X c d\gamma,\tag{P}
\]
and
\[
 \min_{\gamma \in \Pi(\mu^1,\ldots,\mu^N)}C_\infty [\gamma]:= \min_{\gamma \in \Pi(\mu^1,\ldots,\mu^N)} \gamma-\esup_{(x^1,\ldots,x^N)\in X} c,\tag{P$_\infty$}
\]
 in the set 
\[\Pi(\mu^1,\ldots,\mu^N):=\{\gamma\in\mathcal{P}(X)~|~\pi^i(\gamma)=\mu^i\text{ for all }i=1,\ldots,N\},\]
that is, in the set of \emph{couplings} or \emph{transport plans} between the $N$ marginals $\mu^1,\ldots,\mu^N$. 
We refer to the second problem as {\it the $\sup$ case.} The first of these problems is widely encountered in the literature of the last thirty years. The second, although also old, gained popularity only more recently thanks to the applications of optimal transportation in machine learning (see, for example \cite{peyre2019computational}).

This paper is concerned with an optimality condition for the problems above, introduced in the next subsection. 
In particular we will study the sufficiency of such optimality condition. We will give a new, easier, and in our opinion easier-to-understand proof of some known results, and we will show that this new approach allows to extend sufficiency results to a wider setting.  

\subsection{$c$-cyclical monotonocity, $\infty$-$c$-cyclical monotonicity, and the main theorem} 
In this context the $c$-cyclical monotonicity takes the following form.

\begin{definition}\label{cm}
We say that a set $\Gamma\subset\prod_{i=1}^NX^i$ is $c$-cyclically monotone (CM), if for every $k$-tuple of points $(x^{1,i},\ldots, x^{N,i})_{i=1}^k$ and every $(N-1)$-tuple of permutations $(\sigma^2,\ldots,\sigma^N)$ of the set $\{1,\ldots, k\}$ we have
\[\sum_{i=1}^kc(x^{1,i} ,x^{2,i},\ldots,x^{N,i})\le \sum_{i=1}^kc(x^{1,i},x^{2,\sigma^2(i)},\ldots,x^{N, \sigma^N(i)})\,.\]
We also say that $\gamma\in \Pi(\mu^1,\ldots,\mu^N)$ is  $c$-cyclically monotone if it is concentrated on a $c$-cyclically monotone set. 
\end{definition}

\begin{definition}\label{icm}
We say that a set $\Gamma\subset\prod_{i=1}^NX^i$ is infinitely $c$-cyclically monotone (ICM), if for every $k$-tuple of points $(x^{1,i},\ldots, x^{N,i})_{i=1}^k$ and every $(N-1)$-tuple of permutations $(\sigma^2,\ldots,\sigma^N)$ of the set $\{1,\ldots, k\}$ we have
	\[\max\{c(x^{1,i} ,x^{2,i},\ldots,x^{N,i})~|~i\in \{1,\ldots,k\}\}\le \max\{c(x^{1,i},x^{2,\sigma^2(i)},\ldots,x^{N, \sigma^N(i)})~|~i\in \{1,\ldots,k\}\}\,.\]
	We also say that a coupling $\gamma\in \Pi(\mu^1,\ldots,\mu^N)$ is infinitely cyclically monotone if it is concentrated on an ICM set. 
\end{definition}

We will use the expression $c$-cyclically monotone for both conditions above.
The main theorem of this paper is the following

\begin{theorem}\label{icmoptimal}
	Let  $\mu^i \in \Pb(X^i)$ with compact support for $i=1,\dots,N$, let $c:X \to \R\cup\{+\infty\}$ be continuous. If $\gamma\in\Pi(\mu^1,\ldots,\mu^N)$ is an ICM plan for $c$, then $\gamma$ is optimal. 
\end{theorem}	

Along the way we will give a new proof of the following known result due, in a more general setting, to Griesler \cite{griessler2018}.

\begin{theorem}\label{ccmoptimal}
	Let  $\mu^i \in \Pb(X^i)$ with compact support for $i=1,\dots,N$, let $c:X \to \R$ be continuous. If $\gamma\in\Pi(\mu^1,\ldots,\mu^N)$ is an CCM plan for $c$, then $\gamma$ is optimal. 
\end{theorem}

We now discuss an important characterization of $c$-cyclically monotone transport plans. To this aim we define

\begin{definition}\label{finop}
	Let $\gamma$ be a positive and finite Borel measure on $X$. We say that $\gamma$ is finitely optimal if all its finitely-supported submeasures are optimal with respect to their marginals. Here by submeasure we mean any probability measure $\alpha$ satisfying $\supp(\alpha)\subset \supp(\gamma)$. 
\end{definition}

\begin{proposition}\label{icmfinitelyoptimal}
	If $\gamma\in\Pi(\mu^1,\ldots, \mu^N)$ is CM or ICM, then it is finitely optimal for the problem $(P)$ or $(P_\infty)$, respectively. 
\end{proposition}

\begin{lemma}\label{permutations}
	Let $\alpha=\sum_{i=1}^lm_i\delta_{( x^{1,i},\ldots, x^{N,i})}$ and $\overline\alpha=\sum_{i=1}^{\overline l}\overline{m}_i\delta_{(\xbar^{1,i},\ldots,\xbar^{N,i})}$ be two discrete measures with positive, integer coefficients and the same marginals. Let us denote by $\tilde l=m^1 + \dots +m^l$ the number of rows of the following table 
$$	
\left. \begin{array}{lll}  
x^{1,1} & \dots & x^{N,1} \\
\vdots &   & \vdots \\
x^{1,1} & \dots  & x^{N,1}
\end{array}
\right \rbrace
\begin{array}{l}
\\
m^1 \mbox{- times}\\
\\
\end{array} 
$$
$$ \begin{array}{llllllllll}  
\dots & \dots & \dots & & & & & & & 
\end{array}
$$ 
$$
\left.\begin{array}{lll}  
x^{1,l} & \dots & x^{N,l} \\
\vdots &   & \vdots \\
x^{1,l} & \dots  & x^{N,l}
\end{array}
\right \rbrace
\begin{array}{l}
\\
m^l \mbox{- times}\\
\\
\end{array} 
$$
where the first $m^1$ rows are equal among themselves, the following $m^2$ rows are equal among themselves and so on.
Let $\overline A$ be the analogous table associated to $\overline \alpha$. Then $\overline A$ has $\tilde l$ rows and  there exist $(N-1)$ permutations of the set $\{1,\ldots,\tilde l\}$ such that 
$\overline{A}$ is equal to
$$	
\begin{array}{lll}  
x^{1,1} & \dots & x^{N,\sigma^N(1)} \\
\vdots &   & \vdots \\ 
x^{1,1} & \dots & x^{N,\sigma^N (m_1)}\\
x^{1,2} & \dots & x^{N, \sigma^N(m_1 +1)}\\
\vdots &   & \vdots \\ 
x^{1,l} & \dots & x^{N,\sigma^N(m_1+\dots+m_{l-1}+1)}\\
\vdots &   & \vdots  \\ 
x^{1,l} & \dots & x^{N, \sigma^N (\tilde l)}
\end{array} 
$$
\end{lemma}

\begin{proof}
	For each $k\in\{1,\ldots, N\}$, the $k$-th marginal of $\alpha$ is given by the sum of the Dirac masses centered on the points of the $k$-th column of the table $A$ with multiplicity. Analogously, the $k$-th marginal of $\overline \alpha$ is given by the sum of the Dirac masses centered on the points of the $i$-th column of the table $\overline A$ with multiplicity. Since the marginals of $\alpha$ and $\overline \alpha$ are the same, each point $x^{k,i}$ appearing in the  $k$-th marginal must appear in both matrices the same number of times, proving the existence of the bijections $\sigma^2,\ldots,\sigma^N$ as required. This also implies that $\overline A$ has $\tilde l$ rows.  
\end{proof}

\begin{proof}{(of Proposition \ref{icmfinitelyoptimal})} We fix a finitely-supported submeasure $\alpha=\sum_{i=1}^l a_i\delta_{X^i}$ of $\gamma$. We need to show that $\alpha$ is an optimal coupling of its marginals. To do this, we fix another coupling, $\overline\alpha= \sum_{i=1}^{\overline l} \overline a _i\delta_{\overline X^i}$, with the same marginals as $\alpha$. We have to show that 
\begin{equation}\label{eq:costclaim}
\tilde C[\alpha]\le \tilde C[\overline \alpha],
\end{equation}
where $\tilde C$ is any of the two costs under consideration. 
Let us first assume that the discrete measures $\alpha$ and $\overline \alpha$ have rational coefficients. We consider the measures $M\alpha$ and $M \overline \alpha$, where $M$ is the product of the denominators of the coefficients of $\alpha$ and $\overline \alpha$. They are discrete measures having positive, integer coefficients and the same marginals, so we can apply Lemma \ref{permutations} to find permutations $\sigma^2,\ldots,\sigma^N$ such that $M \alpha$ and $M \overline \alpha$ have representations $A$ and $\overline A$, respectively. If $\tilde C=C$ we have, using the $c$-cyclical monotonicity of $\alpha$
\[MC[\alpha]=\sum_{i=1}^{\tilde l}  c(x^{1,i},\ldots, x^{N,i})\le \sum_{i=1}^{\tilde l} c(x^{1,i},x^{2,\sigma^2(i)},\ldots, x^{N, \sigma^N(i))})=MC[\overline \alpha],\]
proving the optimality of $\alpha$. 
If $\tilde C=C_\infty$, the conclusion is immediate:
\[C_\infty[\alpha]=\max_{1\le i\le \tilde k}c(x^{1,i},\ldots, x^{N,i})\le \max_{1\le i\le \tilde k}c(x^{1,i},x^{2,\sigma^2(i)},\ldots, x^{N,\sigma^N(i)})= C_\infty[\overline \alpha].\]

Now, assume that $\alpha$ and $\overline \alpha$ have real (not necessarily rational) coefficients,
\[\alpha:= \sum_{i=1}^l a_i \delta_{X^i}, \ \overline \alpha= \sum_{i=1}^{\overline l} \overline a_i \delta_{\overline X ^i}.
\] 
We show that for all $\varepsilon >0$ there exist two discrete measures 
\[\beta:= \sum_{i=1}^l q_i \delta_{X^i} \ \mbox{and} \  \ \overline \beta= \sum_{i=1}^{\overline l} \overline q_i \delta_{\overline X ^i},
\] 
with the same marginals, $q_i, \overline q_i \in \Q$ and 
\[|a_i-q_i| <\varepsilon, \ |\overline a_i-\overline q_i| <\varepsilon.
\]
Being concentrated on $X^1,\dots,X^l$ and $\overline X^1, \dots \overline X^{\overline l}$ is equivalent to the fact that the vector ${\bf \underline a}:=(a_1, \dots, a_l, \overline a_1, \dots, \overline a_{\overline l})$ is a solution of 
\[\mathcal A {\bf \underline a}=0,\]
where $\mathcal A$ is a matrix with coefficients $1, 0, -1$.
Indeed, if we write, for example, the equality between the first two marginals we obtain
\[ \sum_{i=1}^l a_i \delta_{x^{1,i}}= \sum_{i=1}^{\overline l} \overline a_i \delta_{\overline x ^{1,i}}.
\] 
so some of the points $\overline x^{1,i}$ must coincide with, for example, $x^{1,1}$ and this gives, for two sets of indices
\[\sum_{i\in I} a_i =  \sum_{j\in J} \overline a_j.\] 
Since the matrix $\mathcal A$ has integer coefficients
\[\overline{Ker_\Q \mathcal A} = Ker_\R \mathcal A,\] 
and this allows to choose $\beta$ and $\overline \beta$.
Since $C[\alpha]\approx C[\beta]$, $C[\overline \alpha]\approx C[\overline \beta]$, $C_\infty [\alpha]=C_\infty [\beta]$ and $C_\infty[ \overline \alpha]=C[\overline \beta]$, we conclude.
\end{proof} 
	
\section{Preliminary results}
\subsection{Lower semi-continuity, compactness and existence of minimizers}
Existence for the optimal transport problems above is usually obtained by the direct method of the Calculus of Variations. Here we shortly report the tools which we don't find elsewhere or that will be used substantially in our proofs.
A useful convergence on the set of transport plans is the tight convergence.
\begin{definition} Let $X$ be a metric space and let $\gamma_n \in \Pb (X)$ we say that $\gamma_n$ converges tightly to $\gamma$ if for all $\phi \in C_b (X)$
	\[ \int \phi d\gamma_n \to \int \phi d \gamma.\] 
The tight convergence will be denoted by $\wstar$. 
\end{definition}

\begin{definition} Let $\Pi$ be a set of Borel  probability measures on a metric space $X$. We say that $\Pi$ is tight (or uniformly tight) if for all $\varepsilon >0$ there exists $K_\varepsilon \subset X$ compact such that 
	\[\gamma (K_\varepsilon)> 1-\varepsilon \ \mbox{or, equivalently,} \ \gamma (X\setminus K_\varepsilon)\leq\varepsilon \] 
	for all $\gamma \in \Pi$. 
\end{definition}

\begin{theorem}[Prokhorov] Let $X$ be a complete and separable metric space (Polish space). Then $\Pi \subset \Pb (X)$ is tight if and only if it is pre-compact with respect to the tight convergence. 
\end{theorem}

\begin{remark}\hfill
	\begin{enumerate}
	\item The tight convergence is lower semi-continuous on open sets and  upper semi-continuous on closed sets;
	\item If $X$ is complete and separable then if $\Pi$ is a singleton it is always tight.
\end{enumerate}
\end{remark}
The following compactness theorem will be used in this paper.
\begin{theorem} For $i=1, \dots, N$, let $X^i$ be a Polish space. Let $X= X^1 \times \dots \times X^N$. Let $\mcal^i \subset \Pb (X^i)$ be tight for all $i$. Then the set 
	\[\Pi=\{\gamma \in \Pb (X)\ | \ \pi^i_\sharp \gamma \in \mcal^i  \}\]
is tight.
\end{theorem}
\begin{proof} Let $\varepsilon >0$. By the tightness of $\mcal^i$ we can fix a compact set $K^i \subset X^i$ such that for all $\mu^i \in \mcal^i$
\[\mu^i (X^i \setminus K^i)< \frac\varepsilon N .\]
Let $K= K^1 \times \dots \times K^N$ and let $\gamma \in \Pi$. Since all the marginals $\pi^i_\sharp \gamma \in \mcal ^i$, and since 
\[ X\setminus K \subset ((X^1\setminus K^1 )\times \prod_{k=2}^NX^k)\cup (X^1\times (X^2\setminus K^2 )\times\prod_{k=3}^NX^k)\cup\cdots\cup (\prod_{k=1}^{N-1}X^k\times(X^N\setminus K^N)),
\]
one gets
\[
\gamma(X\setminus K)\leq \varepsilon.
\]
\end{proof}
\begin{corollary}\label{precompatto}
By Prokhorov's theorem a set $\Pi\subset \Pb(X)$ as in the theorem above is pre-compact for the tight convergence. This is, in particular, true if $\mcal^i=\{\mu^i\}$. 
\end{corollary}
 
If $c$ is lower semi-continuous, then also the functionals $C$ and  $C_\infty$ are lower semi-continuous with respect to the tight convergence of measures. The lower semi-continuity of $C$ is a standard result of optimal transport theory (see, for example, \cite{villani} or \cite{kellerer1984} for the multi-marginal case). The next lemma proves the lower semi-continuity of $C_\infty$. 
\begin{lemma}
If the function $c:X\to\R\cup\{+\infty\}$ is lower semi-continuous, then also the functional $C_\infty$ is lower semi-continuous. 
\end{lemma}
\begin{proof}
First we note that, thanks to the lower semi-continuity of $c$, its $\gamma$-essential supremum can be written as
\[\gamma-\esup c=\sup \{c(x^1,\ldots,x^N)~| ~(x^1,\ldots,x^N)\in \supp \gamma\}.\]
Fix  $\gamma\in\Pi(\mu^1,\ldots,\mu^N)$ and let $(\gamma^n)_n$ be a sequence converging to $\gamma$. Now there exist a vector $v\in\supp\gamma$ and a sequence $v^n=(x^{1,n},\ldots,x^{N,n})\in\prod_{i=1}^NX^i$ such that $v^n\in\supp\gamma^n$ for all $n$ and $v^n\to v$. Moreover
\[\liminf_{n\to\infty}C_\infty[\gamma^n]\ge \liminf_{n\to\infty}c(v^n)\ge c(v)\,.\]
Since the above inequality holds for all $v\in\supp\gamma$ and for all sequences converging to $v$, it also holds for the $\gamma$-essential supremum, and the claim follows.
\end{proof}

The use of compactness and semi-continuity theorems above gives the existence of optimal transport plans for both problems considered here.
	
\subsection{$\Gamma$-convergence}	
A crucial tool that we will use in this paper is $\Gamma$-convergence. 
All the details can be found, for instance, in Braides's book \cite{braides2002gamma} or in the classical book by 
Dal Maso \cite{dal1993introduction}. In what follows, $(X,d)$ is a metric space or a topological space equipped with a convergence.

\begin{definition} Let $(F_n)_n$ be a sequence of functions $X \mapsto \bar\R$. We say that $(F_n)_n$ $\Gamma$-converges to 
$F$ if for any $x \in X$ we have
\begin{itemize}
\item for any sequence $(x^n)_n$ of $X$ converging to $x$
$$ \liminf\limits_n F_n(x^n) \geq F(x) \qquad \text{($\Gamma$-liminf inequality);}$$
\item there exists a sequence $(x^n)_n$ converging to $x$ and such that
$$ \limsup\limits_n F_n(x^n) \leq F(x) \qquad \text{($\Gamma$-limsup inequality).} $$
\end{itemize} 
\end{definition}

This definition is actually equivalent to the following equalities for any $x \in X$:
\[
F(x) = \inf\left\{ \liminf\limits_n F_n(x^n) : x^n \to x \right\} = \inf\left\{ \limsup\limits_n F_n(x^n) : x^n \to x \right\} 
\]
The function $x \mapsto  \inf\left\{ \liminf\limits_n F_n(x^n) : x^n \to x \right\}$ is called $\Gamma$-liminf of the sequence $(F_n)_n$ and the other one its $\Gamma$-limsup. A useful result is the following (which for instance implies that a constant sequence of functions does not $\Gamma$-converge to itself in general).

\begin{proposition}\label{lsc} 
The $\Gamma$-liminf and the $\Gamma$-limsup of a sequence of functions $(F_n)_n$ are both lower semi-continuous on $X$. 
\end{proposition}

The main interest of $\Gamma$-convergence resides in its consequences in terms of convergence of minima.

\begin{theorem}\label{convminima} 
Let $(F_n)_n$ be a sequence of functions $X \to \bar\R$ and assume that $F_n$ $\Gamma$-converges to $F$. Assume moreover that there exists a compact and non-empty subset $K$ of $X$ such that
$$ \forall n\in \N, \; \inf_X F_n = \inf_K F_n $$
(we say that $(F_n)_n$ is equi-mildly coercive on $X$). Then $F$ admits a minimum on $X$ and the sequence $(\inf_X F_n)_n$ converges to $\min F$. Moreover, if $(x_n)_n$ is a sequence of $X$ such that
$$ \lim_n F_n(x_n) = \lim_n (\inf_X F_n)  $$
and if $(x_{\phi(n)})_n$ is a subsequence of $(x_n)_n$ having a limit $x$, then $ F(x) = \inf_X F $. 
\end{theorem}

\section{Discretisation of transport plans (Dyadic-type decomposition in Polish spaces)}\label{dyaddecomp}	

Let $\gamma$ be a Borel probability measure on $X=(X^1,d_1)\times \cdots\times (X^N,d_N)$ with marginals $\mu^1,\ldots,\mu^N$. 
The space $X$ will be equipped with the $\sup$ metric  
\[
d(w,z)=\max_{1\le i\le N}d_i(w^i,z^i).
\]
Let $\varepsilon_n=\frac 1n$. Since  $\{\mu^i\}_{i=1}^N$ are Borel probability measures, they are inner regular. Hence for all $n$ there exist compact sets $K^{1,n}\subset\supp\mu^1, K^{2,n}\subset\supp\mu^2,\ldots, K^{N,n}\subset\supp\mu^N$ such that 
\begin{equation}\label{frominside}
\mu^k(X^k\setminus K^{k,n})<\tfrac{\varepsilon_n}{N},
\end{equation}
for all $k=1,\ldots,N.$
We may assume that, for all $k$ and $n$, $K^{k,n}\subset K^{k,{n+1}}$. \\
We denote $K^{n}:=\prod_{k=1}^N K^{k,n}$. Since 
\[ X\setminus K^n \subset ((X^1\setminus K^{1,n})\times \prod_{k=2}^NX^k)\cup (X^1\times (X^2\setminus K^{2,n} )\times\prod_{k=3}^NX^k)\cup\cdots\cup (\prod_{k=1}^{N-1}X^k\times(X^N\setminus K^{N,n})),
\]
one gets
\[
\gamma(X\setminus K^n)\leq \varepsilon_n.
\]
The cost $c$ is uniformly continuous on each $K^{n}$, and for all $n$ we can fix $\delta_n\in (0,\varepsilon_n)$ such that the sequence ($\delta_n$) is decreasing in $n$ and
\[
|c(u)-c(z)|<\varepsilon_n ~~~\text{ for all }u,z\in K^n\text{ for which } d(u,z)<\delta_n.
\]
Next we fix, for all $n$, finite Borel partitions for the sets $K^{1,n},\ldots,K^{N,n}$. We denote these by $\{\tilde B_i^{k,n}\}_{i=1}^{\tilde m^{k,n}}$,  $k=1,\ldots,N$, and we choose them in such a way that for all $n\in\N$ and $k\in\{1,\ldots,N\}$
\[
diam(\tilde B_i^{k,n})<\tfrac12\delta_n,
\]
for all $i\in\{1,\ldots,\tilde m^{k,n}\}$.

We form a new, possibly finer, partition $\{B_i^{k,n}\}_{i=1}^{m^{k,n}}$ for each $K^{k,n}$ by intersecting (if the intersection if nonempty) each element $\tilde B_i^{k,n}$ successively first with the set $K^{k,1}$, then with $K^{k,2}$, and so on up until intersecting with the set $K^{k,n-1}$.  So that for $j \in \{1, \dots n\}$ either $B_i^{k,n} \cap K^{k,j}$ it's empty or it is the entire  $B_i^{k,n}$.
The products
\[
\mathcal{Q}^n=\{ B_{i_1}^{1,n}\times B_{i_2}^{2,n}\times\cdots\times  B_{i_N}^{N,n},~i_k\in\{1,\ldots, m^{k ,n} \}\text{ for all }k=1,\ldots,N\}
\]
form a partition of the set $K^n$ with
\[diam(B_i^{k,n})<\tfrac12\delta_n,
\]
for all $i\in\{1,\ldots,m^{k,n}\}.$

We denote 
\[I^n=\{(i_1,\ldots,i_N)~| ~\gamma(B_{i_1}^{1,n}\times B_{i_2}^{2,n}\times\cdots\times B_{i_N}^{N,n})>0\},\]
and for all ${\bf i}:=(i_1,\ldots,i_N)\in I^n$ we use the notation $Q_\mathbf{i}^n:= B_{i_1}^{1,n}\times\cdots \times B_{i_N}^{N,n}$. We fix points $z^n _{\bf i}=z_{i_1,\ldots,i_N}^n\in\prod_{k=1}^NB_{i_k}^{k,N}\cap\supp\gamma$ (i.e. $z^n _\mathbf{i}\in  Q^n_\mathbf{i}\cap \supp \gamma$).
We define 
\[
\tilde\alpha^n=\sum_{(i_1,\ldots,i_N)\in I^n}\gamma(B_{i_1}^{1,n}\times\cdots \times B_{i_N}^{N,n})\delta_{z_{i_1,\ldots,i_N}^n}~~~\text{and}~~~\alpha^n=\frac{1}{\gamma(K^n)}\tilde\alpha^n;
\]
since $\tilde\alpha^n(X)=\gamma(K^n)$, the measures $\alpha^n$ are probability measures.  

To each multi-index $\mathbf i=(i_1,\ldots,i_N)$ and thus to each point $z_\mathbf i^n$ correspond $N$ points
\[x_\mathbf{i}^{1,n}\in B_{i_1}^{1,n},\ldots,~x_\mathbf{i}^{N,n}\in B_{i_N}^{N,n},\]
which are "coordinates" in the spaces $X_i$ of $z_\mathbf i ^n$.
The marginals of $\alpha^n$ are supported on the Dirac measures given by these points. 
We denote these marginals by $\mu^{1,n},\ldots, \mu^{N,n}$.
More precisely, they can be described as 
%ANNA1
\begin{equation}\label{marginal}
\mu^{k,n}=\frac{1}{\gamma(K^n)}\sum_{i=1}^{m_{k,n}} \sum_{\stackrel{{\bf i} \in I^n}{i={\bf i}_k}} \gamma(Q_\mathbf{i}) \delta_{x_\mathbf{i}^{k,n}}. 
\end{equation}

\begin{proposition}\label{discretegamma} $\alpha^n\rightharpoonup\gamma$. 
\end{proposition}
\begin{proof}
Let $\varepsilon>0$ and $\varphi\in C_b(X)$. 
We have to find $n_0\in\N$ such that
\begin{equation}\label{claim}
|\int_X \varphi d \gamma- \int_X\varphi d \alpha^n | <\varepsilon, ~~~\text{ for all }n\ge n_0.
\end{equation}
Let  $M>0$ be such that
\[
|\varphi(z)| \le M~~~\text{for all }z\in X.\]
We fix $\bar n\in\N$ such that
\[
\gamma(X\setminus K^n)<\min\left\{\frac{1}{2},\frac{\varepsilon}{5M}\right\},~~~\text{for all }n\ge \bar n
\]
Since $\varphi\in C_b(X)$, it is uniformly continuous on the set $K^{\bar n}$, there exists $\delta>0$ such that 
\[
|\varphi(z)-\varphi(v)|<\frac{\varepsilon}{5}~~~\text{ for all }z,v\in K^{\bar n}\text{ such that }d(z,v)<\delta.
\]
Moreover, the decomposition $\mathcal{Q}^n$ has been constructed so that there exists $n_0 \geq \bar{n}$ such that for all $k \in \{1, \dots, N\}$ and $n \geq n_0$
\[
diam(B_i^{k,n})<\delta~~~\text{for all }i\in\{1,\ldots,m^{k,n}\}.
\]
We start from
\begin{equation}\label{total}
|\int_X\varphi \d\gamma-\int_X\varphi \d\alpha^n | \le | \int_{K^{\bar n}}\varphi \d\gamma-\int_{K^{\bar n}}\varphi \d\alpha^n |+ |\int_{X\setminus K^{\bar n}}\varphi \d\gamma-\int_{X\setminus K^{\bar n}}\varphi \d\alpha^n|
\end{equation}
and we evaluate separately the two terms on the RHS. 
For all $n\ge n_0$ the first term can be estimated as follows: (we recall that, by construction, $K^{\bar n}\subset K^n$) 
\begin{align} \label{long}
&| \int_{K^{\bar n}}\varphi\d\gamma-\int_{K^{\bar n}}\varphi\d\alpha^n| 
=|\int_{K^{\bar n}}\varphi\d\gamma-\frac{1}{\gamma(K^n)}\int_{K^{\bar n}}\varphi\d\tilde\alpha^n|\nonumber\\
&\stackrel{a)}{\le}|\int_{ K^{\bar n}}\varphi\d\gamma-\int_{K^{\bar n}}\varphi\d\tilde\alpha^n| +\frac{\gamma(X\setminus K^n)}{1-\gamma(X\setminus K^n)}\int_{K^{\bar n}}|\varphi|\d\tilde\alpha^n\nonumber\\
&<|\int_{K^{\bar n}}\varphi\d\gamma-\int_{K^{\bar n}}\varphi\d\tilde\alpha^n| +M\cdot 2\cdot\frac{\varepsilon}{5M}\nonumber\\
&<|\int_{ K^{\bar n}}\varphi\d\gamma-\int_{K^{\bar n}}\varphi\d\tilde\alpha^n| +\frac{2\varepsilon}{5}.
\end{align} 
Above in $a)$ we have written
\[
\frac{1}{\gamma(K^n)}=\frac{1}{1-\gamma(X\setminus K^n)}=1+\frac{\gamma(X\setminus K^n)}{1-\gamma(X\setminus K^n)}
\]
and then estimated the numerator from above by $\frac{\varepsilon}{5M}$ and the term $\gamma(X\setminus K^n)$ of the denominator from below by $\frac 12$.
By construction, since $n\ge \bar n$, there exist a subset $\bar I^n\subset I^n$ such that 
\[
K^{\bar n}=\bigcup_{(i_1,\ldots, i_N)\in \bar I^n}B_{i_1}^{1,n}\times\cdots\times B_{i_N}^{N,n}.
\]
So we write
\[
\int_{K^{\bar n}}\varphi\d\gamma-\int_{K^{\bar n}}\varphi\d\tilde\alpha^n=\sum_{\mathbf{i}\in \bar I^n}\left(\int_{(B_{i_1}^{1,n}\times\cdots\times B_{i_N}^{N,n})}\varphi\d\gamma-\int_{(B_{i_1}^{1,n}\times\cdots\times B_{i_N}^{N,n})}\varphi\d\tilde\alpha^n\right).
\]

We simplify the notations for the next few lines and, for all $\mathbf{i}\in \bar I^n$, we denote by  $Q:=B_{i_1}^{1,n}\times\cdots\times B_{i_N}^{N,n}$ and by $u_0=z_{i_1,\ldots,i_N}\in Q$ the point in which $\tilde \alpha ^n$ is concentrated.  
Then for each 'cube' $Q$ 
\begin{align*}
&|\int_{Q}\varphi(u)\,\d\gamma -\int_{Q}\varphi(u)\,\d\tilde\alpha^n|=|\int_{Q}\varphi(u)\,\d\gamma-\varphi(u_0)\gamma(Q)|\\
&\le\int_{Q}|\varphi(u)-\varphi(u_0)|\,\d\gamma \leq \gamma(Q)\cdot\frac{\varepsilon}{5},
\end{align*}
and in the last passage we have used the uniform continuity of $\varphi$ on $K^{\bar n}$. 
Summing the estimate above over all cubes $Q=B_{i_1}^{1,n}\times\cdots\times B_{i_N}^{N,n}$, $\mathbf{i}\in \bar I^n$, gives
\[
|\int_{K^{\bar n}} \varphi \d\gamma - \int_{K^{\bar n}} \varphi \d\tilde\alpha^n |< \gamma(K^{\bar n})\cdot\frac{\varepsilon}{5} \leq \frac{\varepsilon}{5}.
\]
Combining this estimate with \eqref{long} gives us the estimate
\begin{equation}\label{int1}
|\int_{K^{\bar n}}\varphi\d\gamma-\int_{K^{\bar n}}\varphi\d\alpha^n|<\frac15 \varepsilon+\frac{2\varepsilon}{5}=\frac35 \varepsilon.
\end{equation}
Finally, the 'tail' term in (\ref{total}). Using the set $\bar I^n$ defined above one gets
\begin{align*}
\alpha^n(X\setminus K^{\bar n})&=1-\frac{1}{\gamma(K^n)}\sum_{(i_1,\ldots,i_N)\in \bar I^{n}}\gamma(B_{i_1}^{1,n}\times\cdots\times B_{i_N}^{N, n})\\
&=1-\frac{\gamma(K^{\bar n})}{\gamma(K^n)}\le 1-\gamma(K^{\bar n})<\frac{\varepsilon}{5M}.
\end{align*}
Using this we get
\begin{align}
&|\int_{X\setminus K^{\bar n}}\varphi\d\gamma-\int_{X\setminus K^{\bar n}}\varphi\d\alpha^n|\le \int_{X\setminus K^{\bar n}}|\varphi|\d\gamma+\int_{X\setminus K^{\bar n}}|\varphi|\d\alpha^n\nonumber\\
&<M\frac{\varepsilon}{5M}+M\frac{\varepsilon}{5M}=\frac25\varepsilon.\label{int2}
\end{align}
Together estimates (\ref{int1}) and (\ref{int2}) prove the claim (\ref{claim}). 
\end{proof}
\begin{remark}\label{discretegamma2}
	If $\supp \mu^k$ is compact for $k=1, \dots, N$ then the dependence on $n$ of $K^n$ is not needed anymore since one can take $K^n\equiv K:= \supp \mu^1 \times \dots \times \supp \mu^N$. 
This also simplifies the analytic expressions of $\alpha^n$ and their marginal measures. 
\end{remark}
In line with the previous Remark we prove the following: 
\begin{proposition}\label{discrecompact}
If $\supp \mu^k$ is compact for $k=1, \dots, N$ then for all $k,n$ and all $i$ 
\[ 
\mu^{k,n} (B_i^{k,n})= \mu^k (B_i^{k,n})
\]
\end{proposition}
\begin{proof}
Again we prove the formula for the first marginal. 
\begin{align*} 
\alpha^n ( B_i ^{1,n}\times\prod_{k=2}^N X^k ) &= \sum_{\stackrel{\iub \in I^n}{i=i_1}} \gamma (Q_\iub ^n) \delta_{z_{\iub}^n} (B_i ^{1,n}\times\prod_{k=2}^N X^k)\\
	&=\sum_{\stackrel{\iub \in I^n}{i=i_1}} \gamma (Q_\iub ^n)=\gamma( B_i ^{1,n}\times\prod_{k=2}^N X^k).
\end{align*}
\end{proof}

\section{Variational approximations and conclusions}
In this section we prove the discrete approximations of the functionals that will be used in the optimality proofs. 
Given a transport plan $\gamma$, we have introduced, in the previous section, the dyadic approximation $\{\alpha^n\}_{n \in \N}$ of $\gamma$.
\subsection{The $\sup$ case.}\label{sup}
We define the functionals $\mathcal{F}_n,\mathcal{F}:\mathcal{P}(X)\to \R \cup\{+\infty\}$ by 
\[
\mathcal{F}_n(\beta)=
\begin{cases}
C_\infty[\beta]&\text{ if }\beta\in\Pi(\mu^{1,n},\ldots,\mu^{N,n}),\\
+\infty&\text{ otherwise; }
\end{cases}
\]
and
\[
\mathcal{F}(\beta)=
\begin{cases}
C_\infty[\beta]&\text{ if }\beta\in\Pi(\mu^1,\ldots,\mu^N),\\
+\infty&\text{ otherwise.}
\end{cases}
\]
For the rest of this subsection we assume that $c$ is continuous and that $\mu^i$ has compact support for $i=1,\dots,N.$ 
We prove the following
\begin{proposition}\label{gammaconv} The functionals $\mathcal{F}_n $ are equi-coercive and 
\begin{equation}\label{eq:gammaclaim} 
\mathcal{F}_n \stackrel{\Gamma}{\rightarrow}\mathcal{F}.
\end{equation}
\end{proposition}
\begin{proof}
Let  $\beta\in\mathcal{P}(X)$. We recall that we need to prove the following: 
\begin{equation*}
\forall (\beta^n)_n\wstar\beta \text{ in }\mathcal{P}(X), \ 
\liminf_{n\to\infty}\mathcal{F}_n(\beta^n)\ge \mathcal{F}(\beta).\tag{I}
\end{equation*}
\begin{equation*}
\exists (\beta^n)_n \wstar \beta \text{ in } \mathcal{P}(X)
\text{ s.t. } \limsup_{n\to\infty}\mathcal{F}_n(\beta^n)\le \mathcal{F}(\beta).\tag{II}
\end{equation*} 
If $\mathcal F[\beta]<+\infty$, the $\Gamma$-$\liminf$ inequality (Condition (I)) follows from the lower-semicontinuity of the functional $C_\infty$. 
If $\mathcal{F}[\beta]=+\infty$, then either $\beta\notin\Pi(\mu^1,\ldots,\mu^N)$ or $C_\infty (\beta)=+\infty$. In the first case, since $\beta^n \wstar \beta$ and  
$\mu^{i,n}\wstar \mu^i$ for $i=1, \dots, N$, there exists $n_0\in\N$ such that  $\beta^n\notin \Pi(\mu^{1,n},\ldots,\mu^{N,n})$ for all $n\ge n_0$. Hence 
$\mathcal{F}_n[\beta^n]=+\infty$ for all $n\ge n_0$. If $C_\infty (\beta)=+\infty$  then let $M>0$ and let $\bf x \in \spt \beta$  and $r>0$ be such that $B({\bf x}, r) \subset \{c>M-\varepsilon\}$. 
Since the evaluation on open sets is lower semi-continuous with respect to the tight convergence we have that, for $n$ big enough,  $\beta_n (B({\bf x}, r))>0$ so that 
$C_\infty (\beta_n)>M-\varepsilon$ and since $M$ is arbitrary we conclude. 

For the $\Gamma$-$\limsup$ inequality  (Condition (II)), if $\mathcal{F}[\beta]=+\infty$, then any sequence with the right marginals and tightly converging to $\beta$ will do. Therefore, we may assume that the measure $\beta$ satisfies $\beta\in\Pi(\mu^1,\ldots,\mu^N)$ and $C_\infty[\beta]<+\infty$. 
To build the approximants, we use the Borel partitions $\{B_i^{k,n}\}_{i=1}^{m^{k.n}}$ and discrete measures introduced in Section \ref{dyaddecomp}.
For all $n$, given a multi-index $\iub=(i_1,\ldots,i_N)$ we use, again, the "cube"
\[
Q_\iub^n:=B_{i_1}^{1,n}\times\cdots\times B_{i_N}^{N,n} \]
and set 
\[J^n:=\{\iub \ | \ \beta (Q^n_\iub) >0\}.\]
We then define the measures
\[
\beta^n=\sum_{\iub \in J^n}  \beta(Q_\iub^n) \frac{\mu^{1,n} \restr{B^{1,n}_{i_1}} }{\mu^1 (B^{1,n}_{i_1})}\otimes \dots \otimes  \frac{\mu^{N,n} \restr{B^{N,n}_{i_N}}}{\mu^N (B^{N,n}_{i_N})}.
\]
We show that $\beta^n$ has marginals $\mu^{1,n},\ldots,\mu^{N,n}$. For all Borel sets $A\subset X_1$ we have 
\begin{align} 
\beta^n\left(A\times\prod_{k=2}^N X_k\right)&=\sum_{\jj\in J^n}\beta (Q_\jj^{n})\frac{\mu^{1,n}_{|B_{j_1}^{1,n}}(A)}{\mu^{1,n}(B_{j_1}^{1,n})}\nonumber \\
&=\sum_{j_1\in\pi^1(J^n)}\frac{\mu^{1,n}_{|B_{j_1}^{1,n}}(A)}{\mu^{1}(B_{j_1}^{1,n})}\sum_{\{(j_2,\ldots,j_N)~|~\jj\in J^n\}}\beta(Q_\jj^n)\nonumber\\
&=\sum_{j_1\in\pi^1(J^n)}\frac{\mu^{1,n}_{|B_{j_1}^{1,n}}(A)}{\mu^{1}(B_{j_1}^{1,n})}\mu^1(B_{j_1}^{1,n})\nonumber \\
&=\sum_{j_1\in\pi^1(J^n)}\mu^{1,n}_{|B_{j_1}^{1,n}}(A)=\mu^{1,n}(A).
\end{align}
where the third inequality is due to Proposition \ref{discrecompact}. The computation is analogous for the other marginals.

The sequence $(\beta^n)$ converges tightly to $\beta$ which can be seen in a manner analogous to the convergence of the sequence $(\alpha^n)$ to $\gamma$.  It remains to prove that the sequence  satisfies the $\Gamma$-$\limsup$ inequality.  
We fix $\varepsilon>0$. It suffices to show that
\[\limsup_{n\to\infty}C_\infty[\beta^n]\le C_\infty[\beta]+\varepsilon.\]
Since for all $n$ the support of $\beta^n$ is a finite set, we can fix $u^n\in\supp\beta^n$ such that $C_\infty[\beta^n]=c(u^n)$. 
Moreover for all $n$ there exists $z^n\in\supp\beta$ such that $d(u^n,z^n)\le\tfrac12\delta_n$.
Now for all $n$ large enough to satisfy $\varepsilon_n<\varepsilon$ we have 
\[C_\infty[\beta^n]=c(u^n)\le c(z^n)+\varepsilon_n\le C_\infty[\beta]+\varepsilon_n<C_\infty[\beta]+\varepsilon\]
and we are done. 

By Corollary \ref{precompatto},  $\Pi (\mu^1, \dots, \ \mu^N)\cup_n \Pi (\mu^{1,n}, \dots, \ \mu^{N,n})$ is compact and therefore the  equi-coercivity follows.
\end{proof}

\subsection{The integral case.}\label{integral}
We define the functionals $\mathcal{G}_n,\mathcal{G}:\mathcal{P}(X)\to \R \cup\{+\infty\}$ by 
\[
\mathcal{G}_n(\beta)=
\begin{cases}
C[\beta]&\text{ if }\beta\in\Pi(\mu^{1,n},\ldots,\mu^{N,n}),\\
+\infty&\text{ otherwise; }
\end{cases}
\]
and
\[
\mathcal{G}(\beta)=
\begin{cases}
C[\beta]&\text{ if }\beta\in\Pi(\mu^1,\ldots,\mu^N),\\
+\infty&\text{ otherwise.}
\end{cases}
\]
For this integral case we assume that the measures $\mu^1,\ldots,\mu^N$ have compact supports and that the cost function $c:X\to\R$ is continuous. 
We prove the following:
\begin{proposition}\label{gammaconvint} The functionals $\mathcal{G}_n $ are equi-coercive and 
\begin{equation}\label{eq:gammaclaimint} 
\mathcal{G}_n \stackrel{\Gamma}{\rightarrow}\mathcal{G}.
\end{equation}
\end{proposition}
\begin{proof}
The proof is analogous to that of Proposition \ref{gammaconv}.  The only substantial difference is in the proof of the $\Gamma$-$\limsup$ inequality in the case that the measure $\beta$ belongs to the set $\Pi(\mu^1,\ldots,\mu^N)$. We have to find a sequence $(\beta^n)$, weakly${}^\ast$-converging to $\beta$ and satisfying Condition  (II). Let  ($\beta^n$) be the discretization defined in the proof of Proposition \ref{gammaconv}. Since the supports of the measures $\mu^1,\ldots,\mu^N$ are compact, also the set $K:=\spt\mu^1\times\cdots\times \spt\mu^N$ is compact. Note that for all $n\in\N$ we have $\spt\beta^n\subset K$. We set 
$T=\max_{z\in K}c(z)$. Now the function $c_T:=\min\{c,T\}$ is continuous and bounded on $X$, and by the weak${}^\ast$-convergence
\[\mathcal{G}(\beta^n)=\int_Xcd\beta^n=\int_X c_Td\beta^n\to\int_Xc_Td\beta=\int_Xcd\beta=\mathcal{G}[\beta],\]
from which the $\Gamma$-$\limsup$ inequality follows. 
\end{proof}

\subsection{Proof of the main theorems and a counterexample.}
\begin{proof}{(of Theorem \ref{icmoptimal})}
By Proposition \ref{discretegamma} and Remark \ref{discretegamma2} we can find a sequence $(\alpha^n)_n$ with finite supports such that $\spt \alpha^n \subset \spt \gamma$ and $\alpha^n \wstar \gamma$.
We define the functionals $\mathcal{F}$ and $\mathcal{F}_n$ of Subsection \ref{sup} using the marginals of $\gamma$ and $\alpha^n$.  
The plan $\gamma$ is ICM, therefore by Proposition \ref{icmfinitelyoptimal} it is finitely optimal. This means that each plan $\alpha^n$ is optimal between its marginals and thus a minimiser of the functional $\mathcal{F}_n$. 

The $\Gamma$-convergence and equi-coercivity established in Proposition \ref{gammaconv} imply,  by Theorem \ref{convminima}, that the minimisers of the functionals $\mathcal{F}_n$ converge, up to subsequences, to a minimiser of $\mathcal F$. Therefore, since $\alpha^n\wstar\gamma$, the plan $\gamma$ is optimal for the problem ($P_\infty$). 
 \end{proof}

\begin{proof}{(of Theorem \ref{ccmoptimal})}
The proof is the same as that of Theorem \ref{icmoptimal}. The $\Gamma$-convergence is now given by Proposition \ref{gammaconvint}. 
\end{proof}

In \cite{Ambrosio}, Ambrosio and Pratelli showed that for the problem ($P$) the continuity of the cost function $c:X\times X\to[0,\infty]$ is necessary to guarantee the optimality of a cyclically monotone transport plan. The next example, that is a slightly modified version of theirs, shows that also in the case of the problem ($P_\infty$) the continuity of the cost is required,  even when the cost assumes only finite values. 
\begin{example}
	Let us consider the two-marginal $L^\infty$-optimal transportation problem with marginals $\mu=\nu=\mathcal{L}|_{[0,1]}$ and the cost function
	\[c(x,y)=
	\begin{cases} 
	1&\text{ if }x=y\\
	2&\text{ otherwise}
	\end{cases}.\]
\end{example}
We fix an irrational number $\alpha$. We set $T_1=Id_{[0,1]}$ and $T_2:[0,1]\to[0,1]$, $T_2(x)=x+\alpha\pmod 1$. Now $T_1$ is an optimal transportation map for the problem ($P_\infty$) with $C_\infty[T_1]=1$. Since $C_\infty[T_2]=2$, $T_2$ cannot be optimal. However, it is ICM. 

In fact if we assume that $T_2$ is not ICM, we should find a minimal $K\in\N$ and a $K$-tuple of couples $\{x_i,y_i\}_{i=1}^K $, all belonging to the support of the plan given by $T_2$, such that 
\begin{equation*}
\max_{1\le i\le K}c(x_i,y_i)>\max_{1\le i\le K}c(x_{i+1},y_i),
\end{equation*}
with the convention $x_{K+1}=x_1$. 
By the definition of the map $T_2$ we have $y_i=x_i+\alpha\pmod 1$ for all $i$. 
Given the form of $c$, the only form in which this inequality can hold is: $2>1$. The right-hand side  now tells us that $y_i=x_i+\alpha\pmod 1$ for all $i$, that is, $x_{i+1}=x_i+\alpha\pmod 1$ for all $i$. Summing up now gives us (keeping in mind that $x_{K+1}=x_1$) that $x_1=x_1+K\alpha\pmod 1$, contradicting the irrationality of $\alpha$.

%%%%%%%%%%%%%%%%%%%%%%%%%%%%%%%%%%%%%%%%%%
%%%%%%%%%%%%%%%%%%%%%%%%%%%%%%%%%%%%%%%%%%

\bibliographystyle{plain}

\end{document}